\documentclass{article}
 
 \usepackage{amsmath}
\usepackage{amsthm}
 \usepackage{amssymb}
\usepackage[pdftex,colorlinks=true,urlcolor=blue,linkcolor=blue,plainpages=false,pdfpagelabels
]{hyperref}

\newtheorem{theorem}{Theorem}[section]

\newtheorem{proposition}[theorem]{Proposition}
\newtheorem{lemma}[theorem]{Lemma}

\def\R{{\mathbb R}}
\def\N{{\mathbb N}}
\DeclareMathOperator*{\argmin}{argmin}

\newcounter{ct}

\begin{document}

\title{Barycenters in uniformly convex geodesic spaces}
\author{Lauren\c{t}iu Leu\c{s}tean${}^{1,2}$,  Adriana Nicolae${}^{3}$, Alexandru Zaharescu${}^{2,4}$ \\[0.2cm]
\footnotesize ${}^1$ Faculty of Mathematics and Computer Science, University of Bucharest,\\
\footnotesize Academiei 14,  P. O. Box 010014, Bucharest, Romania\\[0.1cm]
\footnotesize ${}^2$ Simion Stoilow Institute of Mathematics of the Romanian Academy,\\
\footnotesize P. O. Box 1-764, RO-014700 Bucharest, Romania\\[0.1cm]
\footnotesize ${}^3$ Department of Mathematics, Babe\c{s}-Bolyai University, \\
\footnotesize  Kog\u{a}lniceanu 1, 400084 Cluj-Napoca, Romania\\[0.1cm]
\footnotesize ${}^4$ Department of Mathematics, University of Illinois at Urbana-Champaign,\\
\footnotesize 1409 W. Green Street, Urbana, IL 61801, USA\\[1mm]
\footnotesize E-mails:  laurentiu.leustean@unibuc.ro, anicolae@math.ubbcluj.ro, zaharesc@illinois.edu
}

\date{}

\maketitle

\begin{abstract}
This note proves a result on the existence of barycenters in a class of uniformly convex geodesic spaces.
\end{abstract}

\section{Introduction}

Let $(X,d)$ be a metric space. For $1\leq \theta< \infty$, denote by $\mathcal{P}^\theta(X)$ the set of all probability measures $P$ on $(X,\mathcal{B}(X))$ which satisfy
\begin{equation}
\int_X d^\theta(x, y)d(Py) <\infty\quad \text{for some (and hence any) } x\in X.\label{def-Ptheta}
 \end{equation}
We will mainly be interested in the cases $\theta=1$ and $\theta=2$.

Let $P\in \mathcal{P}^1(X)$ and fix $a\in X$. Following \cite{Stu03}, define the function
\[f_a:X\to \R, \quad f_a(x)=\int_X \big(d^2(x,y)-d^2(y,a)\big)d(Py).\]
The fact that $\displaystyle\int_X \big(d^2(x,y)-d^2(y,a)\big)d(Py)<\infty$ for all $x\in X$ is immediate, hence $f_a$ is well-defined. Furthermore, $f_a$ is continuous. 

In \cite{Stu03}, Sturm proved the following result.

\begin{theorem}\cite[Proposition 4.3]{Stu03}\label{thm-bary-Sturm}\\
Let $(X,d)$ be a complete CAT(0) space and $P\in \mathcal{P}^1(X)$. Then for all $a\in X$, the function $f_a$ has a unique minimum point which does not depend on $a$, is called the {\em barycenter} of $P$ and is denoted by $b(P)$. Thus, 
\[b(P)=\argmin_{x\in X}\int_X \big(d^2(x,y)-d^2(y,a)\big)d(Py).\]
If $P\in \mathcal{P}^2(X)$, then 
\[b(P)=\argmin_{x\in X}\int_X d^2(x,y)d(Py).\]
\end{theorem}

Barycenters in geodesic spaces have been studied by various authors assuming different regularity conditions on the space. For instance, Ohta \cite{Oht12} considered proper Alexandrov spaces of curvature bounded below, while Kell \cite{Kell14} and Kuwae \cite{Kuw14} imposed certain uniform convexity assumptions. Other notions of barycenters and applications thereof to ergodic theory were given, for example, by Austin \cite{Aus11} and Navas \cite{Nav13}.

In this note we prove that Theorem \ref{thm-bary-Sturm} can be extended to the context of geodesic spaces satisfying a uniform convexity condition which is more general than the ones considered in \cite{Kell14,Kuw14}.

\section{Preliminaries}

Let $(X,d)$ be a metric space.  A {\it geodesic} in $X$ is a mapping  $\gamma:[0,l]\to X$ satisfying $d(\gamma(s),\gamma(t))=|s-t|$
for all $s,t\in [0,l]$. A {\it geodesic segment} in $X$ is the image $\gamma([0,l])$ of a geodesic $\gamma:[0,l]\to X$. If 
$\gamma(0)=x$ and $\gamma(l)=y$, we say that the geodesic $\gamma$ or that the geodesic segment $\gamma([0,l])$ {\it joins x and y}. $X$ is said to be a {\it (uniquely) geodesic space} if every two points are joined by a (unique) geodesic. For any $x,y \in X$, a point  belongs to a geodesic segment that joins $x$ and $y$ if and only if there exists $t \in [0,1]$ such that $d(z,x) = td(x,y)$ and $d(z,y) = (1-t)d(x,y)$. A midpoint of $x$ and $y$ is a point, denoted by $m(x,y)$, satisfying $\displaystyle d(m(x,y),x)=d(m(x,y),y)=(1/2) d(x,y)$. Note that if $X$ is a complete metric space for which every two points have a midpoint, then $X$ is a geodesic space. See, for instance, \cite{BriHae99} for more details on geodesic spaces and the notions discussed below.

In the rest of the paper we assume that $(X,d)$ is a uniquely geodesic space, even if not mentioned explicitly.  It follows that every two points $x,y$ of $X$ have a unique midpoint $m(x,y)$.

We say that the metric $d:X \times X \to \R$ is {\it convex} if
\[\quad  d(m(x,y),a) \leq \frac12\left(d(x,a)+d(y,a)\right) \quad\text{ for all } a,x,y\in X.\]
In this case $X$ is also called a geodesic space with convex metric.

A nonempty subset $C\subseteq X$ is said to be {\em convex} if $m(x,y)\in C$ for all $x, y\in C$. If $C$ is a convex set, then a mapping $f:C\to \R$ is {\it quasi-convex} if 
\[ f(m(x,y))\le \max\{f(x),f(y)\}\quad \text{for all }x, y\in C.\]
If strict inequality holds above for any $x, y \in C$, $x \ne y$, then $f$ is {\it strictly quasi-convex}.

$X$ is said to be {\it reflexive} if the intersection of any descending sequence of nonempty, bounded, closed and convex subsets of it is nonempty.  

$X$ is {\it uniformly convex}  \cite{Leu07} if there exists $\eta:(0,\infty)\times(0,2]\rightarrow (0,1]$  such that for any $r>0$ and $\varepsilon\in(0,2]$  and 
for all $a,x,y\in X$,
\begin{equation}\label{eq-def-uc}
\left.\begin{array}{l}
d(x,a)\le r\\
d(y,a)\le r\\
d(x,y)\ge\varepsilon r
\end{array}
\right\}
\quad \Rightarrow  \quad d\left(m(x,y),a\right)\le (1-\eta(r,\varepsilon))r. 
\end{equation}
Such a mapping $\eta$ is called a {\em modulus of uniform convexity}.  We say that $\eta$ is {\it monotone} (resp. {\it lower semi-continuous from the right}) if for every fixed $\varepsilon$ it decreases (resp. is lower semi-continuous from the right) with respect to $r$. Note that any complete uniformly convex geodesic space which admits a monotone (or lower semi-continuous from the right) modulus of uniform convexity is reflexive (see \cite{KohLeu10}). Moreover, one can define uniform convexity in geodesic spaces without assuming a priori uniqueness of geodesics between any two points by supposing that \eqref{eq-def-uc} holds for all midpoints of $x$ and $y$. However, it is easy to see that this implies that the space is uniquely geodesic. A discussion on other particular notions of uniform convexity in metric spaces can be found in \cite{LeuNic15}.

Uniformly convex geodesic spaces are a natural generalization of both uniformly convex Banach spaces and CAT(0) spaces. 
In fact, as pointed out in \cite{Leu07},  these spaces  admit moduli of uniform convexity that do not depend on $r$. Another class of uniformly convex geodesic spaces are the so-called {\em uniform Busemann spaces} defined by Jost \cite[Definition 2.2.6, p.50]{Jos97}. Recall that a {\it Busemann space} is a geodesic space $(X,d)$ satisfying
\[d\left(\gamma_1(l_1/2),\gamma_2(l_2/2)\right)\le\frac12 \left(d\left(\gamma_1(0),\gamma_2(0)\right) + d\left(\gamma_1(l_1),\gamma_2(l_2)\right)\right)\]
for all geodesics $\gamma_1 : [0,l_1] \to X$ and $\gamma_2 : [0,l_2] \to X$. Any Busemann space is uniquely geodesic and its metric is convex. We refer to \cite{Pap05} for an extensive study of these spaces. A Busemann space $(X,d)$ is said to be {\it uniform} if there exists a strictly increasing function $\alpha:[0,\infty)\to[0,\infty)$ with $\alpha(0)=0$ such that for all $a,x,y\in X$,
\[d^2\left(m(x,y),a\right) \leq \frac12 d^2(x,a)+\frac12 d^2(y,a)-\alpha(d(x,y)).\]
It is well-known that any CAT$(0)$ space is a uniform Busemann space. One can easily see that any uniform Busemann space is a uniformly convex geodesic space with a lower semi-continuous from the right
modulus of uniform convexity given by $\eta(r,\varepsilon):=\alpha(\varepsilon r)/(2r^2)$.

\subsection{Some technical results}

Let  $(X,d)$ be a uniquely geodesic space with convex metric. Define
\[S:X^3\to [0,\infty), \quad S(a,x,y)=\frac12d^2(x,a)+\frac12d^2(y,a)-d^2(m(x,y),a).\]
Note that $S$ is nonnegative by the convexity of the metric. For $r>0$ and $\varepsilon\in(0,2]$, define
\[A_{r,\varepsilon}  = \{(a,x,y)\in X^3\mid d(x,a)\leq r, d(y,a)\leq r \text{ and } d(x,y)\geq \varepsilon r\}\]
and 
\[\Phi(r,\varepsilon)=\inf\{S(a,x,y) \mid (a,x,y)\in A_{r,\varepsilon}\} \ge 0.\]
Hence, for all $(a,x,y)\in A_{r,\varepsilon}$, 
\begin{equation}\label{ineq-uc-Phi}
d^2(m(x,y),a)\leq \frac12d^2(x,a)+\frac12d^2(y,a)-\Phi(r,\varepsilon).
\end{equation}

\begin{lemma}\label{prop-inf-Phi}
Let $0<r\leq s$ and $0<\varepsilon\leq \delta\leq 2$. Then
\begin{enumerate}
\item \label{basic-Areps-Phireps-1} $A_{r,\delta}\subseteq A_{r,\varepsilon}$, hence $\Phi(r,\varepsilon)\leq \Phi(r,\delta)$.
\item \label{basic-Areps-Phireps-2} $ A_{r,\varepsilon}\subseteq A_{s,\frac{\varepsilon r}s}$, hence $\Phi(r,\varepsilon)\geq \Phi\left(s,\frac{\varepsilon r}s\right)$.
\item \label{basic-Areps-Phireps-3} Let $I=[r,s] \times [\varepsilon,2]$. Then 
$\Phi(r_1,\varepsilon_1)\geq \Phi\left(s,\frac{\varepsilon r}{s}\right)$ for all $(r_1,\varepsilon_1)\in I$.
\end{enumerate}
\end{lemma}
\begin{proof}
\eqref{basic-Areps-Phireps-1} and \eqref{basic-Areps-Phireps-2} are immediate, so we only prove \eqref{basic-Areps-Phireps-3}. Let $(r_1,\varepsilon_1)\in I$. Since $r_1\leq s$, we apply \eqref{basic-Areps-Phireps-2} to conclude that 
$\displaystyle \Phi(r_1,\varepsilon_1)\geq \Phi\left(s,\frac{\varepsilon_1 r_1}{s}\right)$. Since $r\leq r_1$ and 
$\varepsilon\leq \varepsilon_1$, we have that $\displaystyle \frac{\varepsilon_1 r_1}{s}\geq \frac{\varepsilon r}{s}$. 
Applying now \eqref{basic-Areps-Phireps-1} 
we get that $\displaystyle \Phi\left(s,\frac{\varepsilon_1 r_1}{s}\right)\geq \Phi\left(s,\frac{\varepsilon r}{s}\right)$. Thus, $\displaystyle \Phi(r_1,\varepsilon_1)\geq \Phi\left(s,\frac{\varepsilon r}{s}\right)$.
\end{proof}

In \cite[Theorem 2.3]{KhaKan11}, the function $\Phi$ has been studied for a special class of uniformly convex geodesic spaces, but the proof goes through unchanged in our more general setting. In particular, we get the following result.

\begin{proposition}\label{ucmpm-M1-Phireps-geq0}
Let $X$  be a uniformly convex geodesic space with convex metric. Then $\Phi(r,\varepsilon)>0$ for all  $r>0, \varepsilon\in(0,2]$.
\end{proposition}

\section{Main result}

In this section we prove the following generalization of Theorem \ref{thm-bary-Sturm}.

\begin{theorem}\label{barycenters-ucmpm-P1}
Let $(X,d)$ be a complete uniformly convex geodesic space with convex metric which admits a monotone or lower semi-continuous from the right modulus of uniform convexity and  let $P\in \mathcal{P}^1(X)$. Then for all $a\in X$, the function
\begin{equation}
f_a:X\to \R, \quad f_a(x)=\int_X \big(d^2(x,y)-d^2(y,a)\big)d(Py) \label{def-f-P1X-barycenter-uc}
\end{equation}
has a unique minimum point which does not depend on $a$, is called the \emph{barycenter} of $P$ and 
is denoted by $b(P)$. If $P\in \mathcal{P}^2(X)$, then 
\[
b(P)=\argmin_{x\in X}\int_X d^2(x,y)d(Py) 
\]
\end{theorem}

The main instrument in obtaining the unique minimum point of $f_a$ given by \eqref{def-f-P1X-barycenter-uc} is the following result (see also \cite[Proposition 2.3]{Leu10}). For completeness, we briefly sketch its proof.

\begin{proposition}\label{prop-reflexive-min}
Let $C$ be a nonempty, closed and convex subset of a reflexive geodesic space $X$, $f:C\to\R$ be quasi-convex 
and lower semi-continuous. Assume moreover that for all sequences $(x_n)$ in $C$,
\begin{equation}
\text{if~} \displaystyle \lim_{n\to\infty} d(x_n,p)=\infty \text{~for some~} p\in X,  \text{~then~} (f(x_n))  \text{ is not bounded above}.
\label{eq-prop-reflexive-min}
\end{equation}
Then $f$ attains its minimum on $C$. If, in addition, $f$ is strictly quasi-convex, then $f$ attains its minimum at exactly one point.
\end{proposition}
\begin{proof}
Let $\displaystyle \alpha = \inf_{x \in C} f(x)$ and $(\alpha_n)$ be a strictly decreasing sequence of real numbers which tends to $\alpha$. For $n \in \N$, define $C_n = \left\{x \in C : f(x) \le \alpha_n\right\}$. One can see that $(C_n)$ is a decreasing sequence of nonempty, bounded, closed and convex subsets of $X$. Thus, $\bigcap_{n \in \N}C_n \ne \emptyset$, so there exists $x^* \in C$ with $f(x^*) = \alpha \in \R$. Uniqueness of the minimum point follows immediately when $f$ is strictly quasi-convex.
\end{proof}

\subsection{Proof of Theorem \ref{barycenters-ucmpm-P1}}

We remark first that if a minimum point of $f_a$ exists, then it is independent of $a$, since for arbitrary $b\in X$, the function $f_a - f_b$ is constant. 

In order to get the existence of a unique minimum point of $f_a$, we use Proposition \ref{prop-reflexive-min}. It is easy to see that $f_a$ is 
continuous and that \eqref{eq-prop-reflexive-min} holds. 
It remains to show that $f_a$ is strictly quasi-convex. Fix $x_0,z_0 \in X$ with $x_0 \ne z_0$. For all $y\in X$, let $r(y):=\max\{d(x_0,y),d(z_0,y)\}>0$ and 
$\displaystyle\varepsilon(y):=\frac{d(x_0,z_0)}{r(y)}\leq 2$.  We get that 
\begin{align*}
f_a(m(x_0,z_0)) &= \int_X \left(d^2(m(x_0,z_0),y)-d^2(y,a)\right)d(Py)\\
&\le  \int_X \left(\frac12 d^2(x_0,y)+\frac12 d^2(z_0,y)-
\Phi(r(y),\varepsilon(y))-d^2(y,a) \right) d(Py) \\
& \quad \text{since } (y,x_0,z_0)\in A_{r(y),\varepsilon(y)}, \text{ so we can apply  \eqref{ineq-uc-Phi}}\\
&= \frac{1}{2}\left(f_a(x_0)+f_a(z_0)\right) - \int_X \Phi(r(y),\varepsilon(y)) d(Py)\\
&\leq \max\{f_a(x_0),f_a(z_0)\} - \int_X \Phi(r(y),\varepsilon(y)) d(Py).
\end{align*}
Since $\displaystyle X=\bigcup_{r>0} B(a,r)$ and $P(X)=1$, there exists $R > 0$ with $P(B(a,R)) > 0$. By Proposition \ref{ucmpm-M1-Phireps-geq0}, $\Phi(r(y),\varepsilon(y))>0$ for all $y\in X$. Hence, to obtain strict quasi-convexity of $f_a$, it is enough to show that 
$$\int_{B(a,R)} \Phi(r(y),\varepsilon(y))d(Py) > 0.$$ 
To this end, let $y\in B(a,R)$. It follows that $d(x_0,y)\leq d(x_0,a)+d(a,y)\leq d(x_0,a)+R$ and, similarly, $d(z_0,y)\leq d(z_0,a)+d(a,y)\leq d(z_0,a)+R$. Denoting 
$$s:=\max\{d(x_0,a),d(z_0,a)\}+R,$$ 
we have that $r(y) \le s$.  Then, for all $y\in B(a,R)$,
\[ 0< \frac{d(x_0,z_0)}2\le r(y)\le s \quad\text{and} \quad 0<\frac{d(x_0,z_0)}{s} \le \varepsilon(y) \le 2.\]
Using Lemma \ref{prop-inf-Phi}.\eqref{basic-Areps-Phireps-3}, we conclude that 
\begin{equation} 
\Phi(r(y),\varepsilon(y))\ge \Phi\left(s,\frac{d^2(x_0,z_0)}{2s^2}\right). \label{phi-ineq-1}
\end{equation}
Since \eqref{phi-ineq-1} holds for all $y\in B(a,R)$, it follows that 
\[\int_{B(a,R)}\Phi(r(y),\varepsilon(y)) d(Py)\ge P(B(a,R))\,\Phi\!\left(s,\frac{d^2(x_0,z_0)}{2s^2}\!\right)>0.\]
Hence, we can apply Proposition \ref{prop-reflexive-min} to get the existence of $b(P)$.

Finally, for $P\in  \mathcal{P}^2(X)$ let $g:X\to [0,+\infty)$ be defined by $$\displaystyle g(x)=\int_X d^2(x,y)d(Py).$$ 
Then $f_a-g$ is constant, so clearly $b(P)$ is the unique minimum point of $g$. $\hfill\Box$

\mbox{ } 

\noindent
{\bf Acknowledgements:} \\[1mm] 
Lauren\c tiu Leu\c stean was supported by a grant of the Romanian 
National Authority for Scientific Research, CNCS - UEFISCDI, project 
number PN-II-ID-PCE-2011-3-0383. \\
Lauren\c tiu Leu\c stean and Alexandru Zaharescu are grateful to Florin Boca for helpful discussions on the subject of this paper.

\end{document}